\documentclass[12pt]{article}
\usepackage{amsfonts,amsmath,amssymb,amsthm,url}
\newtheorem{theorem}{Theorem}[section]

\newtheorem{lemma}[theorem]{Lemma}
\newtheorem{corollary}[theorem]{Corollary}
\newtheorem{con}[theorem]{Conjecture}

\usepackage{hyperref}
\usepackage{makeidx}
\makeindex

\begin{document}

\title{Rational approximants for the Euler-Gompertz constant\footnote{This paper is prepared under the supervision of A. Skopenkov and is submitted to the Moscow Mathematical Conference for High-School Students. Readers are invited to send their remarks and reports on this paper to mmks@mccme.ru}}

\author{Vassily Bolbachan}
\date{\today}
\maketitle

\begin{abstract}
    We obtain two sequences of rational numbers which converge to the Euler-Gompertz
constant.
Denote by $\left<f(x)\right>$ the integral of $f(x)e^{-x}$ from 0 to infinity.
Recall that the Euler-Gompertz constant $\delta$ is $\left<\ln(x+1)\right>$.

{\bf Main idea.} Let $P_n(x)$ be a polynomial with integer coefficients. It is
easy to prove that $\left<P_n(x)\ln(x+1)\right>=a_n+\left<\ln(x+1)\right>b_n$ for some integers $a_n$, $b_n$. Hence if
$\left<P_n(x)\ln(x+1)\right>/b_n$ converges to zero, $a_n/b_n$ converges to
$-\delta$.

{\bf Main Theorem.} {\it Let u be positive real. There exists polynomials $P_n(x)$ (they are explicitly given in
the paper) such that $\left<P_n(x)\ln(xu+1)\right>$ tends to u as n tends to infinity.}

Proof of Main Theorem is elementary.

\end{abstract}

\section{Main result}

\begin{theorem}

\label{Main_Theorem}

For each real $u \ge 0$

$$u=\sum\limits_{m=r}^{\infty}{\sum\limits_{k=r}^{m} {m\choose k} {k\choose r} \dfrac{(-1)^{k+r}}{k!}\int_0^{\infty}{x^{k-1}e^{-x}\ln(xu+1)dx}}.$$

\end{theorem}
We recall that
$$\delta=\int\limits_0^{\infty}\ln(x+1)e^{-x}dx.$$
\begin{corollary}
\label{Corollary1}

Let $r\geqslant 0$ be integer. We define two sequences of integer numbers $a_m$ and $b_m$ by formulas

$$a_m=\sum_{k=r}^m{m\choose k}^2{k\choose r}(m-k)!\sum_{w=0}^{k-1}(-1)^ww!\qquad b_m=\sum_{k=r}^m{m\choose k}^2{k\choose r}(m-k)!.$$

Then $\lim\limits_{m\to\infty}\dfrac{a_m}{b_m}=-\delta.$

\end{corollary}

\begin{corollary}
\label{Corollary2}

Let $r\geqslant 1$ be integer. We define two sequences of integer numbers $a_m$ and $b_m$ by formulas

$$a_m=m!\sum_{k=r}^{m}\sum\limits_{j=0}^{k-1}\sum_{i=0}^{j-1}{m\choose k}{k\choose r}\dfrac{i!}{kj!}(-1)^{k+j+i+1}$$

$$b_m=m!\sum_{k=r}^{m}\sum\limits_{j=0}^{k-1}{m\choose k}{k\choose r}\dfrac{(-1)^{k+j}}{kj!}.$$

Then $\lim\limits_{m\to\infty}\dfrac{a_m}{b_m}=-\delta.$

\end{corollary}

\begin{con}For each real $u>0$

$$\psi(u)=\ln(u)+\lim\limits_{m\to\infty}\sum\limits_{k=1}^mA_{k,m+1}{m\choose k}\dfrac{(-1)^k}{k!m!}\int_0^{\infty}x^{k-1}e^{-x}\ln\left(\dfrac{x+u}{u}\right)$$

where

$$A_{k,m}=\sum\limits_{t=2}^m\left\{{m\atop t}\right\}\sum\limits_{w=1}^{t-1}(-k)^{t-w}\sum\limits_{j=1}^w(-1)^{j}B_j
\left[w \atop j\right].$$
\end{con}

Here $\psi(x)$ is the digamma function, $\left[w \atop j\right]$ is Stirling numbers of the first kind, $\left\{w \atop j\right\}$ is Stirling numbers of the second kind and $B_j$ is the Bernoulli numbers. Definitions can be found in \cite{Ro06}, \cite{Wa03}. See also \cite{GS06}, \cite{Bo10}.

\section{Proof of Theorem \ref{Main_Theorem}}

Let $u\ge 0$ be real and $r\geqslant 0$ be integer. For each real $q>-1$ by definition, put
\begin{equation}
\label{f_k}
f_q(u)={q\choose r}\dfrac 1{\Gamma(q+1)}\int_0^{\infty}x^{q-1}e^{-x}\ln(xu+1)dx
\end{equation}

where $\Gamma(q+1)$ is the Gamma function. (see for example \cite{Wa03}.)

In order to prove Theorem \ref{Main_Theorem} we need

\begin{lemma}
\label{Main_Lemma}
For each real $u\ge 0$ and $\varepsilon\in(-1;-1/2)$ we have
$$\lim\limits_{m\to\infty}\sum\limits_{j=0}^m{m\choose j}(-1)^jf_{\varepsilon+j}(u)=0.$$
For each real $u_0>0$ the limit converges uniformly for $u\in[0;u_0]$ and $\varepsilon\in(-1;-1/2)$.
\end{lemma}

Lemma \ref{Main_Lemma} will be proved below.

\begin{proof}[Proof of Theorem \ref{Main_Theorem}] The proof is in two steps.

\begin{description}
    \item[\it Step 1.] Let us prove that $\lim\limits_{\varepsilon\to-1}f_{\varepsilon}(u)=(-1)^ru$.

We have
$$(-1)^r\lim\limits_{\varepsilon\to -1}f_{\varepsilon}(u)=(-1)^r\lim\limits_{\varepsilon\to -1}{\varepsilon\choose r}\dfrac1{\Gamma(1+\varepsilon)}\int_0^{\infty}x^{\varepsilon-1}\ln(xu+1)e^{-x}dx\overset{(*)}=$$

$$\overset{(*)}=\lim\limits_{\varepsilon\to -1}\dfrac1{\Gamma(1+\varepsilon)}\int_0^{\infty} x^{\varepsilon-1}\ln(xu+1)e^{-x}dx\overset{(**)}=$$

$$\overset{(**)}=\lim\limits_{\varepsilon\to -1}\dfrac1{\Gamma(1+\varepsilon)}\int_0^{\infty} x^{\varepsilon-1}(xu)e^{-x}dx=u\lim\limits_{\varepsilon\to -1}\dfrac1{\Gamma(1+\varepsilon)}\Gamma(1+\varepsilon)=u.$$

The equality (*) follows because

$$\lim\limits_{\varepsilon\to -1}{\varepsilon \choose r}={-1 \choose r}=\dfrac{(-1)(-2)\dots(-r)}{r!}=(-1)^r.$$

The equality (**) follows because
$$\dfrac 1{\Gamma(1+\varepsilon)}=\dfrac{1+\varepsilon}{\Gamma(2+\varepsilon)}.$$
and
$$\left|\int\limits_0^{\infty}x^{\varepsilon-1}(\ln(xu+1)-xu)e^{-x}dx\right|\leq$$
$$\leq\left|\int\limits_0^{1}x^{-2}(\ln(xu+1)-xu)e^{-x}\right|+ \left|\int\limits_1^{\infty}x^{-3/2}(\ln(xu+1)-xu)e^{-x}\right|.$$

    \item[\it Step 2.] By Lemma \ref{Main_Lemma}, we get
$$0=\lim\limits_{m\to\infty}\sum\limits_{j=0}^m{m\choose j}(-1)^jf_{\varepsilon+j}(u)=f_{\varepsilon}(u)+\lim\limits_{m\to\infty}\sum\limits_{j=0}^{m-1}(-1)^{j+1}{m\choose j+1}f_{j+1+\varepsilon}(u).$$

In Step 1 we proved that $f_{\varepsilon}(u)$ tends to $(-1)^ru$ as $\varepsilon$ tends to $-1$. Hence

$$u=\lim\limits_{\varepsilon\to -1}\lim\limits_{m\to\infty}\sum\limits_{j=0}^{m-1}(-1)^{j+r}{m\choose j+1}f_{j+1+\varepsilon}(u).$$

Also by Lemma \ref{Main_Lemma} the convergence in this formula is uniform for $\varepsilon\in(-1;-1/2)$. Hence changing the order of limits and substituting $m+1$ by $m$ in this formula, we get
$$u=\lim\limits_{m\to\infty}\sum\limits_{j=0}^{m}(-1)^{j+r}{m+1\choose j+1}f_{j}(u)=$$

$$=\lim\limits_{m\to\infty}\sum\limits_{j=0}^{m}{m+1\choose j+1}{j\choose r}\dfrac {(-1)^{j+r}}{j!}\int_0^{\infty}x^{j-1}\ln(xu+1)e^{-x}dx.$$

Denote by $S_m$ the expression under the limit. We have

$$S_m-S_{m-1}=\sum\limits_{j=0}^{m}\left({m+1\choose j+1}-{m\choose j+1}\right){j\choose r}\dfrac{(-1)^{j+r}}{j!}\int_0^{\infty}x^{j-1}\ln(xu+1)e^{-x}dx=$$
$$=\sum\limits_{j=0}^{m}{m\choose j}{j\choose r}\dfrac {(-1)^{j+r}}{j!}\int_0^{\infty}x^{j-1}\ln(xu+1)e^{-x}dx.$$

Hence

$$u=\sum_{m=0}^{\infty}(S_m-S_{m-1})=\sum_{m=0}^{\infty}\sum\limits_{j=0}^{m}{m\choose j}{j\choose r}\dfrac{(-1)^{j+r}}{j!}\int_0^{\infty}x^{j-1}\ln(xu+1)e^{-x}dx.$$
\end{description}

\end{proof}

\begin{proof}[Proof of Lemma \ref{Main_Lemma}] Proof is in three steps.

\begin{description}
    \item[\it Step 1.] Let $j>0$ be integer and $\varepsilon\in (-1;-1/2)$ be real. We claim that

\begin{equation}
\label{Different_Equality}
f_{\varepsilon+j}(u)={{\varepsilon+j-r\choose j}}^{-1}\sum\limits_{i=0}^j{\varepsilon+j-1\choose j-i}\dfrac{u^i}{i!}f_{\varepsilon}^{(i)}(u).
\end{equation}

The proof is by induction over $j$. Let us prove the base of induction for $j=1$. We must prove that

\begin{equation}
\label{base}
f_{\varepsilon+1}(u)=\dfrac {\varepsilon}{\varepsilon+1-r}f_{\varepsilon}(u)+\dfrac 1{\varepsilon+1-r}uf_{\varepsilon}'(u).
\end{equation}

Integrating formula (\ref{f_k}) by part, we get

$$f_{\varepsilon}(u)=-{\varepsilon\choose r}\dfrac 1{\Gamma(\varepsilon+1)}\int\limits_0^{\infty}\dfrac {x^{\varepsilon}}{\varepsilon}\left(-e^{-x}\ln(xu+1)+u\dfrac{e^{-x}}{xu+1}\right)dx=$$

\begin{equation}
\label{posle_differicirovania}
=\dfrac{\varepsilon-r+1}{\varepsilon}f_{\varepsilon+1}(u)-\dfrac u{\varepsilon} {\varepsilon \choose r}\dfrac 1{\Gamma(\varepsilon+1)}\int\limits_0^{\infty}x^{\varepsilon}e^{-x}\dfrac {dx}{xu+1}.
\end{equation}


For each real $u_0>0$ integral in formula (\ref{f_k}) converges uniformly for $u\in[0;u_0]$. Hence
differentiating formula (\ref{f_k}) with respect to $u$, we get

$$f_{\varepsilon}'(u)={\varepsilon\choose r}\dfrac 1{\Gamma(\varepsilon+1)}\int_0^{\infty}x^{\varepsilon}e^{-x}\dfrac {dx}{xu+1}.$$

Combining this with formula (\ref{posle_differicirovania}), we obtain

$$f_{\varepsilon}(u)=\dfrac{\varepsilon-r+1}{\varepsilon}f_{\varepsilon+1}(u)-\dfrac u{\varepsilon} f_{\varepsilon}'(u).$$

The base of induction follows.

Let us prove the step of induction. By the inductive hypothesis for $j=N$, substituting $\varepsilon+1$ for
$\varepsilon$, we get

$$f_{\varepsilon+N+1}(u)={{\varepsilon+N+1-r\choose N}}^{-1}\sum\limits_{i=0}^N{\varepsilon+N\choose N-i}\dfrac{u^i}{i!}f_{\varepsilon+1}^{(i)}(u).$$

Substituting formula (\ref{base}) in this formula, we get

$$f_{\varepsilon+N+1}(u)=$$

$$={{\varepsilon+N+1-r\choose N}}^{-1}\sum\limits_{i=0}^N{\varepsilon+N\choose N-i}\dfrac{u^i}{i!}\dfrac{d^i}{du^i}\left(\dfrac {\varepsilon}{\varepsilon+1-r}f_{\varepsilon}(u)+\dfrac 1{\varepsilon+1-r}uf_{\varepsilon}'(u)\right).$$

Or equivalently

\begin{equation}
\label{shag_of_induction} f_{\varepsilon+N+1}(u)(N+1){\varepsilon+N+1-r\choose
N+1}=\sum\limits_{i=0}^N{\varepsilon+N\choose N-i}\dfrac{u^i}{i!}\dfrac{d^i}{du^i}\left(\varepsilon
f_{\varepsilon}(u)+uf_{\varepsilon}'(u)\right).
\end{equation}

If we substituting in the Leibniz formula

$$\dfrac{d^i}{du^i}\left(h(u)g(u)\right)=\sum\limits_{k=0}^i{i\choose k}f^{(k)}(u)g^{(i-k)}(u)$$

$u$ for $h(u)$ and $f'(u)$ for $g(u)$, we obtain

$$\dfrac{d^i}{du^i}(uf_{\varepsilon}'(u))=uf_{\varepsilon}^{(i+1)}(u)+if_{\varepsilon}^{(i)}(u).$$

Hence the right-hand side of formula (\ref{shag_of_induction}) can be rewritten as

$$\sum\limits_{i=0}^N{\varepsilon+N\choose N-i}\dfrac{u^i}{i!}\left(\varepsilon f_{\varepsilon}^{(i)}(u)+uf_{\varepsilon}^{(i+1)}(u)+if_{\varepsilon}^{(i)}(u)\right)=$$$$=\sum\limits_{i=0}^N{\varepsilon+N\choose N-i}\dfrac{u^i}{i!}(\varepsilon+i)f_{\varepsilon}^{(i)}(u)+\sum\limits_{i=0}^N{\varepsilon+N\choose N-i}\dfrac{u^i}{i!}uf_{\varepsilon}^{(i+1)}(u)=$$

$$=\sum\limits_{i=0}^{N+1}\dfrac{u^i}{i!}f_{\varepsilon}^{i}(u)\left({\varepsilon+N\choose N-i}(\varepsilon+i)+{\varepsilon+N\choose N-i+1}i\right).$$

From the formula

$${\varepsilon+N\choose N-i}(\varepsilon+i)+{\varepsilon+N\choose N-i+1}i=(N+1){\varepsilon+N\choose N-i+1}$$

it follows that

$$f_{\varepsilon+N+1}(u)(n+1){\varepsilon+N+1-r\choose N+1}=(n+1)\sum\limits_{i=0}^{N+1}\dfrac{u^i}{i!}f_{\varepsilon}^{i}(u){\varepsilon+N\choose N-i+1}.$$

Dividing both sides by $(n+1){\varepsilon+N+1-r\choose N+1}$, we get formula (\ref{Different_Equality}) for
$j=N+1$. The step of induction follows.

\item[\it Step 2.] Let us prove that

\begin{equation}
\label{bin_formula} \sum\limits_{j=i}^m{m\choose j}{{\varepsilon+j-r\choose j}}^{-1}{\varepsilon+j-1\choose
j-i}(-1)^j={m-i-r\choose m-i}{m+\varepsilon-r\choose m}^{-1}(-1)^i.
\end{equation}

By definition, put

$$F(a,b,c;x)=\sum\limits_{k=0}^{\infty}\dfrac {x^k}{k!}\dfrac{a(a+1)\dots (a+k-1)b(b+1)\dots(b+k-1)}{c(c+1)\dots(c+k-1)}.$$

This series converges, if $|x|\leq 1$ and $a+b<c$.

We have

$$\sum\limits_{j=i}^m{m\choose j}{{\varepsilon+j-r\choose j}}^{-1}{\varepsilon+j-1\choose j-i}(-x)^j=$$$$=\dfrac{m!\Gamma(\varepsilon-r+1)}{(m-i)!\Gamma(\varepsilon-r+i+1)}(-x)^iF(i+\varepsilon,i-m,\varepsilon+i-r+1;x).$$

Let us prove that

$$\dfrac{m!\Gamma(\varepsilon-r+1)}{(m-i)!\Gamma(\varepsilon-r+i+1)}F(i+\varepsilon,i-m,\varepsilon+i-r+1;1)=$$
$$={m-i-r\choose m-i}{m+\varepsilon-r\choose m}^{-1}=\dfrac{(m-i-r)!m!\Gamma(\varepsilon-r+1)}{(m-i)!\Gamma(1-r)\Gamma(m+\varepsilon-r+1)}.$$

Or equivalently

$$F(i+\varepsilon,i-m,\varepsilon+i-r+1;1)=\dfrac{(m-i-r)!\Gamma(\varepsilon-r+i+1)}{\Gamma(1-r)\Gamma(m+\varepsilon-r+1)}.$$

This formula follows by the Gauss's theorem (see \cite[p. 282]{Wa03})

$$F(a,b,c;1)=\dfrac{\Gamma(c-a-b)\Gamma(c)}{\Gamma(c-a)\Gamma(c-b)},$$

for $a=i+\varepsilon, b=i-m$ and $c=\varepsilon+i-r+1$.

Formula (\ref{bin_formula}) is proved.

\item[\it Step 3.] From formula (\ref{Different_Equality}) it follows that

$$\sum\limits_{j=0}^m{m\choose j}(-1)^jf_{\varepsilon+j}(u)=$$

$$=\sum\limits_{j=0}^m{m\choose j}(-1)^j{{\varepsilon+j-r\choose j}}^{-1}\sum\limits_{i=0}^j{\varepsilon+j-1\choose j-i}\dfrac{u^i}{i!}f_{\varepsilon}^{(i)}(u)=$$

$$=\sum\limits_{j=0}^m\sum\limits_{i=0}^j{m\choose j}{{\varepsilon+j-r\choose j}}^{-1}{\varepsilon+j-1\choose j-i}(-1)^j\dfrac{u^i}{i!}f_{\varepsilon}^{(i)}(u)=$$

$$=\sum\limits_{i=0}^m\sum\limits_{j=i}^m{m\choose j}{{\varepsilon+j-r\choose j}}^{-1}{\varepsilon+j-1\choose j-i}(-1)^j\dfrac{u^i}{i!}f_{\varepsilon}^{(i)}(u)=$$

$$=\sum\limits_{i=0}^m\dfrac{u^i}{i!}f_{\varepsilon}^{(i)}(u)\sum\limits_{j=i}^m{m\choose j}{{\varepsilon+j-r\choose j}}^{-1}{\varepsilon+j-1\choose j-i}(-1)^j.$$

Combining this with formula (\ref{bin_formula}), we obtain

\begin{equation}
\label{Simple_Formula}
{m+\varepsilon-r\choose m}^{-1}\sum\limits_{i=0}^m\dfrac{u^i}{i!}f_{\varepsilon}^{(i)}(u)(-1)^i{m-i-r\choose m-i}.
\end{equation}

\item[\it Step 4.] In this step we need

\begin{lemma}
\label{Lemma_about_fluxion}
Let $u\ge 0$ and $\varepsilon\in(-1;-1/2)$ be real and $n>0$ be integer. We have

$$\lim\limits_{m\to\infty}m^n\dfrac{f_{\varepsilon}^{(m)}(u)u^m}{u(1+\varepsilon)m!}=0.$$

For each real $u_0>0$ the limit converges uniformly for $u\in[0;u_0]$ and $\varepsilon \in(-1;-1/2)$.

\end{lemma}

Lemma \ref{Lemma_about_fluxion} will be proved below.

Let us consider two cases.

\begin{description}

\item[\it Case 1 : let $r$ be zero.] Let $x$ and $\theta$ be real and $0 \le x\le u,\theta\in(0;1)$. By the Taylor's theorem in the Cauchy form for the function $f_{\varepsilon}(u)$, we obtain

$$f_{\varepsilon}(x)=\sum\limits_{i=0}^m\dfrac{(x-u)^i}{i!}f_{\varepsilon}^{(i)}(u)+\dfrac{(x-u)^{m+1}(1-\theta)^m}{m!}f_{\varepsilon}^{m+1}(u+\theta(x-u)).$$

Putting in this formula $x=0$, we obtain

$$\sum\limits_{i=0}^m\dfrac{u^i}{i!}f_{\varepsilon}^{(i)}(u)(-1)^i=\dfrac{u^{m+1}(-1)^{m}(1-\theta)^m}{m!}f_{\varepsilon}^{m+1}(u(1-\theta)).$$

Hence using inequality

$${m+\varepsilon \choose m}^{-1}=\dfrac{m!}{(\varepsilon+1)\dots(\varepsilon+m)}=$$

$$=\dfrac{m}{1+\varepsilon}\left(\dfrac{m-1}{\varepsilon+m}\right)\dots\left(\dfrac{1}{\varepsilon+2}\right)<\dfrac{m}{1+\varepsilon}$$

we have

$$\left|{m+\varepsilon\choose m}^{-1}\sum\limits_{i=0}^m\dfrac{u^i}{i!}f_{\varepsilon}^{(i)}(u)(-1)^i\right|<$$
$$<m(m+1)(-1)^mu\dfrac{f_{\varepsilon}^{m+1}(u(1-\theta))(u(1-\theta))^{m+1}}{(u(1-\theta))(1+\varepsilon)(m+1)!}.$$

The left-hand side of this inequality equals expression (\ref{Simple_Formula}), but for each real $u_0>0$ the
right-hand side of this inequality tends to $0$ as $m$ tends to $\infty$ uniformly for $u\in[0;u_0]$ and
$\varepsilon\in(-1;-1/2)$ by Lemma \ref{Lemma_about_fluxion}.

\item[\it Case 2 : let $r$ be positive.] Expression (\ref{Simple_Formula}) can be rewritten as

$${m+\varepsilon-r\choose m}^{-1}\sum\limits_{i=m-r+1}^m\dfrac{u^i}{i!}f_{\varepsilon}^{(i)}(u)(-1)^i{m-i-r\choose m-i}.$$

Because for $m-i-r\geq 0$, we get ${m-i-r\choose m-i}=0$. Let $j\in [0; r-1]$ be integer. We must prove that

$$\lim\limits_{m\to \infty}{m+\varepsilon-r\choose m}^{-1}\dfrac{u^{m-j}}{(m-j)!}f_{\varepsilon}^{(m-j)}(u)(-1)^{m-j}{j-r\choose j}=0.$$

Or equivalently

$$\lim\limits_{m\to \infty}{m+j+\varepsilon-r\choose m+j}^{-1}\dfrac{u^{m}}{m!}f_{\varepsilon}^{(m)}(u)=0.$$

There exists integer number $n$ and real number $C$ such that $\left|{m+j+\varepsilon-r\choose m+j}^{-1}\right|<Cm^n$. Hence

$$\left|{m+j+\varepsilon-r\choose m+j}^{-1}\dfrac{u^{m}}{m!}f_{\varepsilon}^{(m)}(u)\right|<Cm^n\dfrac{u^{m}}{m!}f_{\varepsilon}^{(m)}(u).$$

By Lemma \ref{Lemma_about_fluxion} the right-hand of this inequality tends to $0$ as $m$ tends to $\infty$.

\end{description}
\end{description}
\end{proof}

\begin{proof}[Proof of Lemma \ref{Lemma_about_fluxion}]

For each real $u_0>0$ integral in formula (\ref{f_k}) converges uniformly for $u\in[0;u_0]$. Hence differentiating formula (\ref{f_k}) with respect to u, we get

$$f_{\varepsilon}^{(m)}(u)u^m=(-1)^{m+1}(m-1)!{\varepsilon\choose r}\dfrac 1{\Gamma(1+\varepsilon)}\int_0^{\infty}x^{\varepsilon-1}e^{-x}\dfrac{(xu)^m}{(xu+1)^m}dx.$$

Let $T=\sqrt{m}$ and $m>4$. We have

$$\int_0^{\infty}x^{\varepsilon-1}e^{-x}\dfrac{(xu)^m}{(xu+1)^m}dx=$$

\begin{equation}
\label{summa}
\int_0^Tx^{\varepsilon-1}e^{-x}\left(1+\dfrac 1{xu}\right)^{-m}dx+\int_T^{\infty}x^{\varepsilon-1}e^{-x}\left(1+\dfrac 1{xu}\right)^{-m}dx.
\end{equation}

\begin{description}
  \item[\it The first term.] For each real $x\in[0;T]$, we have
  $$\int_0^Tx^{\varepsilon-1}e^{-x}\left(1+\dfrac 1{xu}\right)^{-m}dx<u^2T^{2+\varepsilon}\left(1+\dfrac 1{Tu}\right)^{2-m}$$
  because
  $$x^{\varepsilon-1}\left(1+\dfrac 1{xu}\right)^{-m}\leq u^2T^{1+\varepsilon}\left(1+\dfrac 1{Tu}\right)^{2-m}\quad \text{and} \quad e^{-x}\leq 1.$$
  \item[\it The second term.] For each real $x\in[T;+\infty)$, we have
    $$\int_T^{\infty}x^{\varepsilon-1}e^{-x}\left(1+\dfrac 1{xu}\right)^{-m}dx<\int_T^{\infty}e^{-x}(xu)dx=ue^{-T}(T+1)$$
    because
  $$x^{\varepsilon-1}\left(1+\dfrac 1{xu}\right)^{-m+1}\dfrac {xu}{1+xu}<xu.$$

\end{description}

Hence

$$\left|m^n\dfrac{f_{\varepsilon}^{(m)}(u)u^m}{u(1+\varepsilon)m!}\right|<\dfrac{(-1)^{m+1}m^{n-1}}{\Gamma(2+\varepsilon)}{\varepsilon\choose r}\left(u\sqrt{m}^{2+\varepsilon}\left(1+\dfrac 1{\sqrt{m}u}\right)^{2-m}+e^{-\sqrt{m}}\left(\sqrt{m}+1\right)\right).$$

Clearly, for each real $u_0>0$ the expression in the right-hand sides tends to $0$ as $m$ tends to $\infty$
uniformly for $u\in[0;u_0]$ and $\varepsilon\in(-1;-1/2)$.

\end{proof}

\section{Proof of Corollary \ref{Corollary1} and Corollary \ref{Corollary2}}

In order to prove Corollary \ref{Corollary1} and Corollary \ref{Corollary2} we need

\begin{lemma}
\label{lemma_about_integral}
For each integer $n\geq 0$
\begin{equation}\label{LemmaA}
    \int_0^{\infty}\dfrac{x^n}{x+1}e^{-x}dx=(-1)^{n}\left(\sum_{j=0}^{n-1}\left(j!(-1)^{j+1}\right)+\delta\right)
\end{equation}
and

\begin{equation}\label{LemmaB}
\int\limits_0^{\infty}x^n\ln(x+1)e^{-x}dx=\sum\limits_{j=0}^n\dfrac{n!}{j!}(-1)^j\left(\sum_{i=0}^{j-1}\left(i!(-1)^{i+1}\right)+\delta\right).
\end{equation}
\end{lemma}

Formula (\ref{LemmaA}) can be found in \cite[f. 3.353.5]{Ry07}, formula (\ref{LemmaB}) can be found in \cite[f. 4.337.5]{Ry07}.


\begin{proof}[Proof of Corollary \ref{Corollary1}]

The formula of Theorem \ref{Main_Theorem} converges uniformly for $u\in[0;1]$. Hence differentiating the formula of Theorem \ref{Main_Theorem} respect with to $u$ and taking $u=1$, we get

$$1=\sum\limits_{m=r}^{\infty}\sum\limits_{k=r}^m{m\choose k}{k\choose r}\dfrac{(-1)^{k+r}}{k!}\int\limits_0^{\infty}\dfrac{x^ke^{-x}}{x+1}dx.$$

Series in the right-hand side of this formula converges. Hence

$$\lim\limits_{m\to\infty}\sum\limits_{k=r}^m{m\choose k}{k\choose r}\dfrac{(-1)^{k+r}}{k!}\int\limits_0^{\infty}\dfrac{x^ke^{-x}}{x+1}dx=0.$$

By formula (\ref{LemmaA}) of Lemma \ref{lemma_about_integral}, we get

$$\sum\limits_{k=r}^m{m\choose k}{k\choose r}\dfrac{(-1)^{k+r}}{k!}\int\limits_0^{\infty}\dfrac{x^ke^{-x}}{x+1}dx=$$$$=
\sum\limits_{k=r}^m{m\choose k}{k\choose r}\dfrac{(-1)^{k+r}}{k!}(-1)^{k}\left(\sum_{j=0}^{k-1}\left(j!(-1)^{j+1}\right)+\delta\right)=$$

$$(-1)^r\sum\limits_{k=r}^m\sum_{j=0}^{k-1}{m\choose k}{k\choose r}\dfrac{j!}{k!}(-1)^{j+1}+(-1)^r\sum\limits_{k=r}^m{m\choose k}{k\choose r}\dfrac1{k!}\delta=$$

$$=\dfrac{(-1)^r}{m!}\sum\limits_{k=r}^m{m\choose k}^2{k\choose r}(m-k)!\sum_{j=0}^{k-1}j!(-1)^{j+1}+\delta\dfrac{(-1)^r}{m!}\sum\limits_{k=r}^m{m\choose k}^2{k\choose r}(m-k)!.$$

Clearly, the expression

$$\sum\limits_{k=r}^m{m\choose k}{k\choose r}\dfrac1{k!}$$

tends to $\infty$ as m tends to $\infty$. Hence

$$\lim\limits_{m\to\infty}\dfrac{\sum\limits_{k=r}^m{m\choose k}^2{k\choose r}(m-k)!\sum\limits_{j=0}^{k-1}j!(-1)^{j+1}}{\sum\limits_{k=r}^m{m\choose k}^2{k\choose r}(m-k)!}=-\delta.$$
\end{proof}

\begin{proof}[Proof of Corollary \ref{Corollary2}]
Series in the right-hand side of the formula of Theorem \ref{Main_Theorem} converges. Hence

$$\lim\limits_{m\to\infty}\left({\sum\limits_{k=r}^{m} {m\choose k} {k\choose r} \dfrac{(-1)^{k+r}}{k!}\int_0^{\infty}{x^{k-1}e^{-x}\ln(xu+1)dx}}\right)=0.$$

Taking $u=1$ in this formula and using formula (\ref{LemmaB}) of Lemma \ref{lemma_about_integral}, we obtain

$$\lim_{m\to\infty}\left[\sum_{k=r}^{m}{m\choose k}{k\choose r}\dfrac{(-1)^{k+r}}{k!}\sum\limits_{j=0}^{k-1}\dfrac{(k-1)!}{j!}(-1)^j\left(\sum_{i=0}^{j-1}\left(i!(-1)^{i+1}\right)+\delta\right)\right]=0.$$

Or equivalently

$$\sum_{k=r}^{m}\sum\limits_{j=0}^{k-1}\sum_{i=0}^{j-1}{m\choose k}{k\choose r}\dfrac{i!}{kj!}(-1)^{k+r+j+i+1}+\delta\sum_{k=r}^{m}\sum\limits_{j=0}^{k-1}{m\choose k}{k\choose r}\dfrac{(-1)^{k+r+j}}{kj!}.$$

We must prove that the expression

$$A_m(r):=r(-1)^{r+1}\sum_{k=r}^{m}\sum\limits_{j=0}^{k-1}{m\choose k}{k\choose r}\dfrac{(-1)^{k+r+j}}{kj!}=\sum_{j=1}^{m}\dfrac{(-1)^{j}}{(j-1)!}\sum\limits_{k=j}^{m}{m\choose k}{k-1\choose r-1}(-1)^i$$

tends to $+\infty$ as $m$ tends to $\infty$. We claim that $A_m(r+2)-A_m(r)$ tends to $+\infty$ as m tends to $\infty$. We have

$$A_m(r)+A_m(r+1)=\sum_{j=1}^{m}\dfrac{(-1)^{j}}{(j-1)!}\sum\limits_{k=j}^{m}{m\choose k}\left({k-1\choose r-1}+{k-1\choose r}\right)(-1)^k=$$
$$=\sum_{j=1}^{m}\dfrac{(-1)^{j}}{(j-1)!}\sum\limits_{k=j}^{m}{m\choose k}{k\choose r}(-1)^k=\sum_{j=1}^m{m\choose j}{j\choose r}\dfrac{j-r}{m-r}\dfrac 1{(j-1)!}$$

because
\begin{equation}\label{BinFormula2}
    \sum\limits_{k=j}^{m}{m\choose k}{k\choose r}(-1)^k={m\choose j}{j\choose r}\dfrac{j-r}{m-r}(-1)^j.
\end{equation}

This formula will be proved below.

Hence for $m>r+1$, we obtain

$$A_m(r+2)-A_m(r)=(A_m(r+1)+A_m(r+2))-(A_m(r)+A_m(r+1))=$$

$$\sum\limits_{j=1}^m{m\choose j}{j \choose r}\dfrac {j-r}{(j-1)!}\left(\dfrac{j-r-1}{(m-r-1)(r+1)}-\dfrac 1{m-r}\right)>$$

$$>\sum\limits_{j=r+1}^{2r+3}{m\choose j}{j \choose r}\dfrac {j-r}{(j-1)!}\left(\dfrac{j-r-1}{(m-r-1)(r+1)}-\dfrac 1{m-r}\right)=\dfrac{P(m)}{m-r-1}$$

Here $P(m)$ is a polynomial, $\deg(P)\ge 2r+2\ge 2$. Hence the right-hand sides tends to $+\infty$ as m tends to $\infty$. But $A_m(0)=0$ and

$$A_m(1)=A_m(0)+A_m(1)=\sum\limits_{j=1}^m{m-1\choose j-1}\dfrac 1{(j-1)!}.$$

Hence $A_m(r)$ tends to $+\infty$ as m tends to $\infty$ for each positive integer $r$.

Let us prove formula \ref{BinFormula2}. We have

$$\sum\limits_{k=j}^{m}{m\choose k}{k\choose r}(-1)^k=(-x)^j\dfrac{m!}{(m-j)!r!(j-r)!}F(1,j-m,1+j-r;x)=$$
$$=(-x)^j{m\choose j}{j\choose r}F(1,j-m,1+j-r;x).$$

Hence we must prove that

$$F(1,j-m,1+j-r;1)=\dfrac{j-r}{m-r}.$$

This formula follows by the GaussТs theorem

$$F(a,b,c;1)=\dfrac{\Gamma(c-a-b)\Gamma(c)}{\Gamma(c-a)\Gamma(c-b)}$$

for $a=1,b=j-m,c=1+j-r$.

\end{proof}

{\bf Acknowledgment.} The author is grateful A. Skopenkov for useful comments on this paper and
references.



\it

Vasily Bolbachan, AESC MSU

e-mail address: \url{ys93@bk.ru}.

\printindex

\end{document}